\documentclass[
submission
]{dmtcs-episciences}

\usepackage[utf8]{inputenc}
\usepackage{subfigure}
\usepackage{amsmath}
\usepackage{amssymb}
\usepackage{amsthm}
\usepackage{hyperref}

\usepackage[round]{natbib}

\newtheorem{definition}{Definition}[section]

\newtheorem{lemma}[definition]{Lemma}
\newtheorem{falselemma}[definition]{False lemma}

\newlength{\taille} \makeatletter
\def\qed{%
  \ifmmode\vrule width .5\baselineskip height 0pt depth .5\baselineskip%
  \else{%
    \unskip\nobreak\hfil%
    \setlength{\taille}{\f@size\p@}%
    \penalty50\hskip1em\null\nobreak\hfil\vrule width .5\taille height
    0pt depth .5\taille
    \parfillskip=0pt\finalhyphendemerits=0\endgraf}%
  \fi} \makeatother

\newlength{\taillepreuve}

\newenvironment{contreexemple}{%
  \setbox123=\hbox{Counter-example:}%
  \taillepreuve=\wd123%

  \vskip-\lastskip\nobreak\medskip\par\noindent\box123\list{}{\leftmargin
    .5\taillepreuve}\parindent=1em\item} {\qed\endlist\bigskip}

\DeclareMathOperator{\rank}{rank}

\author{Laurent Lyaudet}
\title{Erratum to ``On Operations and Linear Extensions of Well Partially Ordered Sets''}
\affiliation{\url{https://lyaudet.eu/laurent/}}
\keywords{partial order}
\received{}
\revised{}
\accepted{}
\begin{document}
\maketitle
\begin{abstract}
In this article, we give a counter-example to Lemma 12 of the article
``On Operations and Linear Extensions of Well Partially Ordered Sets''
by Maciej Malicki and Aleksander Rutkowski.
\end{abstract}

Current version : 2019/04/28

\section{Introduction}
\label{section:introduction}

In this article, we give a counter-example to Lemma 12 of the article
``On Operations and Linear Extensions of Well Partially Ordered Sets''
by Maciej Malicki and Aleksander Rutkowski (\cite{DBLP:journals/order/MalickiR04}).

\section{Definitions and notations}
\label{section:definitions_and_notations}

\begin{definition}[Rank function]
Each well-founded poset \begin{math}P\end{math} 
admits an ordinal valued rank function \begin{math}\rank_{P}\end{math}
defined inductively on its elements:
\begin{math}\rank_P(a) = sup_{x <_P a}(\rank_P(x) + 1)\end{math}
\end{definition}

Let \begin{math}\mathcal{P} = \{ P_t : t \in T\}\end{math} be an ordered family of ordered sets,
i.e. both \begin{math}P_t\end{math}'s and \begin{math}T\end{math} are partially ordered
(by \begin{math}\leq_{t}\end{math} and \begin{math}\leq_{T}\end{math} respectively).
With no loss of generality, elements of \begin{math}\mathcal{P}\end{math} can be assumed to be pairwise disjoint.
Let, for \begin{math}a \in \bigcup_{t \in T} P_t\end{math}, 
\begin{math}f(a)\end{math} be that unique \begin{math}t\end{math}
such that \begin{math}a \in P_t\end{math}.

Now, assume all elements of \begin{math}\mathcal{P}\end{math} to be well-founded
and call, for \begin{math}a \in \bigcup_{t \in T} P_t\end{math},
the \emph{primitive rank} of \begin{math}a\end{math} an ordinal \begin{math}g(a) = \rank_{P_{f(a)}}(a)\end{math}.
Define the following ranked order \begin{math}<_R\end{math} on \begin{math}\bigcup_{t \in T} P_t\end{math}:
\begin{math}a <_R b\end{math} if 
\begin{itemize}
\item either \begin{math}f(a) = f(b)\end{math} and \begin{math}a <_{f(a)} b\end{math},
\item or \begin{math}f(a) < f(b)\end{math} and \begin{math}g(a) \leq g(b)\end{math}.
\end{itemize}
Call the union with that order the \emph{ranked sum} and denote it \begin{math}\mathcal{RP}\end{math}.
Observe that \begin{math}a \leq_R b\end{math} implies \begin{math}f(a) \leq_T f(b)\end{math}.

\section{A counter example to Lemma 12}
\label{section:A_counter_example_to_Lemma_12}

\begin{falselemma}[Lemma 12]
Let both \begin{math}T\end{math} and all the components \begin{math}P_t\end{math}
of \begin{math}\mathcal{RP}\end{math} be well-founded
(hence \begin{math}P = \mathcal{RP}\end{math} is well-founded too).
Then for each \begin{math}a \in \bigcup_{t \in T} P_t\end{math},
\begin{math}\rank_P(a) \leq \rank_T(f(a)) + \rank_{f(a)}(a)\end{math}.
\end{falselemma}

\begin{contreexemple}
It is easy to construct an order \begin{math}\mathcal{RP}\end{math}
with an element \begin{math}a\end{math} such that
\begin{math}\rank_P(a) = \rank_T(f(a)) + \rank_{f(a)}(a) + 1\end{math}.
Indeed consider \begin{math}T = \{0, 1\}\end{math}, and \begin{math}P_0 = P_1 = \omega + 1\end{math}
(\begin{math}\omega\end{math} is the first infinite ordinal).
Let \begin{math}a\end{math} be the maximum of \begin{math}P_1\end{math}, 
and \begin{math}b\end{math} be the maximum of \begin{math}P_0\end{math}.
Then \begin{math}\rank_{\mathcal{RP}}(b) = \rank_{P_0}(b) \end{math},
hence \begin{math}\rank_{\mathcal{RP}}(a) = \rank_{\mathcal{RP}}(b) + 1 = \omega + 1 
> \rank_T(f(a)) + \rank_{f(a)}(a) = \rank_T(P_1) + \rank_{P_1}(a) = 1 + \omega = \omega\end{math}
(ordinal sum is not commutative and \begin{math}1 + \omega \neq \omega + 1\end{math}).

The problem in the proof is in the line 
\begin{math}\rank_T(f(b)) + \rank_{f(b)}(b) + 1 \leq \rank_T(f(a)) + \rank_{f(a)}(a)\end{math}.
It should be corrected to 
\begin{math}\rank_T(f(b)) + 1 + \rank_{f(b)}(b) \leq \rank_T(f(a)) + \rank_{f(a)}(a)\end{math},
but then the proof by transfinite induction fails.

You cannot correct the lemma by switching both ranks, i.e.
\begin{math}\rank_P(a) \leq \rank_{f(a)}(a) + \rank_T(f(a))\end{math}.
Indeed then \begin{math}T = \omega + 1\end{math}, and \begin{math}P_0 = P_1 = ... = P_{\omega} = \{0, 1\}\end{math}
is a counter-example.
\end{contreexemple}

\begin{lemma}
For any ordinal \begin{math}\alpha\end{math},
there is an order \begin{math}\mathcal{RP}\end{math}
with an element \begin{math}a\end{math} such that
\begin{math}\rank_P(a) = \rank_T(f(a)) + \rank_{f(a)}(a) + \alpha\end{math}.
\end{lemma}
\begin{proof}
Consider \begin{math}T = \alpha + 1\end{math}, and \begin{math}P_0 = P_1 = ... = P_{\alpha} = \beta + 1\end{math},
where \begin{math}\beta\end{math} is the first ordinal such that
\begin{math}\alpha + \beta = \beta\end{math}.
Let \begin{math}a\end{math} be the maximum of \begin{math}P_{\alpha}\end{math}, 
and \begin{math}b\end{math} be the maximum of \begin{math}P_0\end{math}.
Then \begin{math}\rank_{\mathcal{RP}}(b) = \rank_{P_0}(b) \end{math},
hence \begin{math}\rank_{\mathcal{RP}}(a) = \rank_{\mathcal{RP}}(b) + \alpha = \beta + \alpha 
> \rank_T(f(a)) + \rank_{f(a)}(a) = \rank_T(P_{\alpha}) + \rank_{P_{\alpha}}(a) = \alpha + \beta = \beta\end{math}.
\end{proof}

\begin{lemma}
For any ordinal \begin{math}\alpha\end{math},
there is an order \begin{math}\mathcal{RP}\end{math}
with an element \begin{math}a\end{math} such that
\begin{math}\rank_P(a) = \rank_{f(a)}(a) + \rank_T(f(a)) + \alpha\end{math}.
\end{lemma}

\section{Conclusion}

We sent an email to one of the authors on 2019/02/24 but, unfortunately, we never had an answer.
We hope this erratum may be useful to the scientific community.

\acknowledgements
\label{section:acknowledgements}

We thank God: Father, Son, and Holy Spirit. We thank Maria.
They help us through our difficulties in life.

\nocite{*}
\bibliographystyle{abbrvnat}
\bibliography{LL2019ErratumMalickiRutkowski2003}
\label{section:bibliography}

\end{document}